\documentclass[letterpaper, 10 pt, conference]{ieeeconf}

\IEEEoverridecommandlockouts                              % This command is only
\usepackage{amsmath}    % need for sub equations
\usepackage{amsfonts}
\usepackage{graphicx}   % need for figures
\usepackage{subcaption}
\usepackage{epsfig} 
\usepackage{color}
\usepackage[normalem]{ulem}
\usepackage{cancel}
\usepackage{amssymb}
\usepackage{color}
\usepackage[ruled,vlined,titlenotnumbered]{algorithm2e} 
\usepackage{algorithmic}
\usepackage{caption}
\usepackage{multirow}
\usepackage{booktabs}
\usepackage{cite}

\newcommand{\R}{\mathbb{R}} % real number
\newcommand{\veh}{Q} % vehicle

\newcommand{\ctrl}{u}
\newcommand{\cset}{\mathcal{U}}
\newcommand{\cfset}{\mathbb{U}}

\newcommand{\targetset}{\mathcal L}

\newcommand{\dz}{\mathcal{Z}} % danger zone

\newcommand{\costate}{\lambda}

\newcommand{\st}{K}
\newcommand{\ic}{l}
\newcommand{\frsvf}{W}
\newcommand{\frs}{\mathcal \frsvf}
\newcommand{\state}{x}
\newcommand{\brsvf}{V}
\newcommand{\brs}{\mathcal \brsvf}
\newcommand{\ham}{H}
\newcommand{\reldyn}{g}
\newcommand{\PCS}{P} % potential conflict set
\newcommand{\Nveh}{\mathcal N}
\newcommand{\OI}{I}
\newcommand{\our}{\mathcal S}
\newcommand{\PCSOIJ}{P_{\OI j}} % potential conflict set with outsider & another vehicle j

\newcommand{\vdyn}{f}

\newtheorem{defn}{Definition}
\newtheorem{prop}{Proposition}
\newtheorem{thm}{Theorem}
\newtheorem{rmk}{Remark}

\title{\LARGE \bf 
A Hybrid Framework for Multi-Vehicle Collision Avoidance}

\author{ Aparna Dhinakaran*, Mo Chen*, Glen Chou, Jennifer C. Shih, Claire J. Tomlin.
\thanks{This research is supported by ONR under the Embedded Humans MURI (N00014-16-1-2206).}
\thanks{* Both Authors contributed equally to this work. All authors are with the Department of Electrical Engineering and Computer Sciences, University of California, Berkeley. \{aparnadhinak,mochen72,gchou,cshih,tomlin\}@berkeley.edu}
}

\begin{document}
\maketitle
\thispagestyle{empty}
\pagestyle{empty}

%%%
\begin{abstract}
With the recent surge of interest in UAVs for civilian services, the importance of developing tractable multi-agent analysis techniques that provide safety and performance guarantees has drastically increased. Hamilton-Jacobi (HJ) reachability has successfully provided these guarantees to small-scale systems and is flexible in terms of system dynamics. However, the exponential complexity scaling of HJ reachability with respect to system dimension prevents its direct application to larger-scale problems where the number of vehicles is greater than two. In this paper, we propose a collision avoidance algorithm using a hybrid framework for $N+1$ vehicles through higher-level control logic given any $N$-vehicle collision avoidance algorithm. Our algorithm conservatively approximates a guaranteed-safe region in the joint state space of the $N+1$ vehicles and produces a safety-preserving controller. In addition, our algorithm does not incur significant additional computation cost. We demonstrate our proposed method in simulation.
\end{abstract}

% !TEX root = nplus1.tex
\section{Introduction}

With the increasingly widespread adoption of unmanned aerial vehicles (UAVs), it is important to understand interactions in multi-vehicle systems; indeed, government agencies such as the Federal Aviation Administration and the National Aeronautics and Space Administration are urgently working on UAV-related regulations \cite{FAA13, Kopardekar16}.

Multi-agent systems are characterized by the asymmetric goals of and cooperation among the involved agents. This setting captures many real-world scenarios, and so these systems has been studied extensively in the past, with a particular focus on collision-avoidance. However, most prior lines of work introduce simplifying assumptions to obtain satisfactory results. For instance, \cite{Fiorini98, Vandenberg08} assume that the agents will employ certain simple control strategies which induce velocity obstacles that must be avoided in order to maintain safety. \cite{Chuang07, Saber02} assume that multiple agents maintain formation to travel along pre-specified trajectories. Similarly, other works such as \cite{Ahmadzadeh2009, Beard2003, Bellingham, Feng-LiLian2002, Frazzoli2002a} require strong assumptions on the dynamics of the system. While these approaches capture important behavior of multi-agent systems for many realistic scenarios, they do not offer general safety guarantees important in safety-critical systems.

Attempts at addressing this more general setting have utilized optimal control and differential game approaches. These are ideal tools due to their flexibility with respect to the system dynamics. In particular, the HJ formulation of differential games has been studied extensively and successfully applied to small-scale problems such as automatic aerial refueling \cite{Ding08}, pairwise collision avoidance \cite{Mitchell05}, reach-avoid games \cite{Huang11, Vaisbord88}, and many others. However, despite the many successes of HJ reachability, the computational complexity of HJ-based methods scales exponentially with the number of vehicles in the system, making their direct application to multi-vehicle problems intractable. Thus, utilizing differential games to analyze larger-scale problems usually requires structural assumptions on the roles of different agents or on the traffic rules of the system. For example, \cite{Fisac15b, Su14, Tanimoto78} discuss various classes of three-player differential game with different assumptions on the role of each agent in non-cooperative settings. For even larger systems, \cite{Chen15, Chen15b, Chen14, Chen17, Lin13} provide promising results when varying assumptions in the form of traffic rules and vehicle priorities can be made.

In this paper, we attempt to resolve the above two issues simultaneously; we attempt to solve the problem of collision avoidance in unstructured flight with low computational overhead. Our proposed algorithm achieves collision avoidance for $N+1$ vehicles via black box use of any algorithm for $N$-vehicle collision avoidance. In particular, our algorithm employs a hybrid systems approach to provide a joint cooperative control strategy and conservatively approximates a guaranteed-safe region in the joint state space of the $N+1$ vehicles. To do this, our algorithm treats one of the vehicles as an ``outsider'' vehicle, and then quantifies the set of safe states for the outsider vehicle given the states of the other $N$ vehicles. To preemptively account for cases in which no safe states can be found, we formulate our solution to allow sufficient time for any vehicle to remove itself from the joint system. 

Our paper is organized as follows:
\begin{itemize}
    \item In Section \ref{sec:formulation}, we formulate the multi-vehicle collision avoidance problem, provide a summary of HJ reachability from \cite{Chen2016, Mitchell05}, and state our goals.
    \item In Section \ref{sec:method}, we describe our hybrid framework for achieving  $N+1$-vehicle collision avoidance given some $N$-vehicle collision avoidance algorithm. We prove that our proposed hybrid framework conservatively approximates the unsafe regions of any vehicle given the other $N$ vehicles, and that the cooperative controllers synthesized by the hybrid framework guarantee safety whenever possible.
    \item Finally, in Section \ref{sec:sim}, we demonstrate our proposed method in a four-vehicle simulation. (Prior HJ reachability-based approaches for guaranteeing safety were limited to two or three vehicles \cite{Chen2016, Mitchell05}).
\end{itemize}

% Introduction (1p)
%% Motivation
%% Related work
%% Summary of results

% !TEX root = nplus1.tex
\section{Problem Formulation}
\label{sec:formulation}

Consider $N+1$ vehicles, denoted $\veh_i, i = 1, 2, \ldots, N+1$, whose dynamics are described by the system of ordinary differential equation (ODE)
\begin{equation}
\label{eq:vdyn} % vehicle dynamics
\dot \state_i = f_i(\state_i, \ctrl_i), \quad \ctrl_i \in \cset_i, \quad i = 1,\ldots, N+1
\end{equation}

\noindent where $\state_i \in \R^{n_i}$ is the state of the $i^\text{th}$ vehicle $\veh_i$, and $\ctrl_i$ is the control of $\veh_i$. Each vehicle $\veh_i$ will have some objective, such as getting to a set of goal states. Whatever the objective may be, each vehicle $\veh_i$ must at all times avoid the \textit{danger zone} $\dz_{ij}$ with respect to each of the other vehicles $\veh_j, j \neq i$. These danger zones represent undesirable relative configurations between $\veh_i$ and $\veh_j$ (such as collisions). %In this paper, we make the assumption that $\state_i \in \dz_{ij} \Leftrightarrow x_j \in \dz_{ji}$, the interpretation of which is that between a pair of vehicles, an unsafe configuration is one in which either of the vehicle is the danger zone of the other.

If possible and desired, each vehicle uses a ``goal-satisfaction controller'' that helps complete its objective. However, sometimes a ``safety controller'' must be used in order to prevent the vehicle from entering danger zones. Since these danger zones $\dz_{ij}$ are sets of joint configurations, it is convenient to derive the set of relative dynamics between every vehicle pair from the dynamics of each vehicle specified in \eqref{eq:vdyn}. Let the relative dynamics between $\veh_i$ and $\veh_j$ be specified by the following ODE.
\begin{equation}
\label{eq:rdyn} % relative dynamics
\begin{aligned}
\dot{x}_{ij} &= g_{ij}(x_{ij}, \ctrl_i, \ctrl_j) \\
\ctrl_i &\in \cset_i, \ctrl_j \in \cset_j \quad i, j = 1, \ldots, N+1, i\neq j 
\end{aligned}
\end{equation}

We assume the functions $f_i$ and $g_{ij}$ are uniformly continuous, bounded, and Lipschitz continuous in arguments $\state_i$ and $x_{ij}$ respectively for fixed $\ctrl_i$ and $(\ctrl_i, \ctrl_j)$ respectively. In addition, the control functions $\ctrl_i(\cdot)\in\cfset_i$ are drawn from the set of measurable functions\footnote{A function $f:X\to Y$ between two measurable spaces $(X,\Sigma_X)$ and $(Y,\Sigma_Y)$ is said to be measurable if the preimage of a measurable set in $Y$ is a measurable set in $X$, that is: $\forall V\in\Sigma_Y, f^{-1}(V)\in\Sigma_X$, with $\Sigma_X,\Sigma_Y$ $\sigma$-algebras on $X$,$Y$.}.

% Our proposed method described in later sections builds on HJ reachability-based methods \cite{Chen2016, Mitchell05} for guaranteeing $N$-vehicle collision avoidance with $N=2$ and $N=3$. We now go over theory and definitions that will allow us to analyze the $N+1$-vehicle collision avoidance problem using our new method.

\subsection{Background}
\subsubsection{Hamilton-Jacobi Reachability\label{sec:HJI}}
The main goal of reachability analysis is to characterize the set of states either from which the system can be driven into some target set, or that the system can reach from some target set. Specifically, in the backward reachability problem, we would like to determine the backward reachable set (BRS) $\brs(T)$ of time horizon $T$ representing states from which the system can be driven into some potentially \textit{time-varying} target set $\targetset(t), t\in[-T, 0]$. In the forward reachability problem, we would like to determine the forward reachable set (FRS) $\frs(T)$ of time horizon $T$ representing states that a system can reach starting from the target set $\targetset$.
% A few different definitions of BRSs and FRSs will be used in our proposed method, and we will present their definitions when needed.
In general, $\brs(t)$ can be obtained from the viscosity solution \cite{Crandall84} to the HJ PDE \eqref{eq:brs_pde}. Although there are many equivalent methods of obtaining the BRS \cite{Bokanowski10, Fisac15, Mitchell05}, we will use the formulation in \cite{Fisac15} which solves the time-varying reachability problem without augmenting the state space with time.
\begin{equation}
    \label{eq:brs_pde}
    \begin{aligned}
    \min \{D_t \brsvf(t,\state) + \ham(\state, \nabla \brsvf(t,\state)), \ic(t, \state) - \brsvf(t, \state) \} = 0 \\ 
    \brsvf(T,\state) = \ic(T, \state) \\
    t\in [0,T]
    \end{aligned}
\end{equation}
\noindent where the target set is represented by the sub-zero level set of the implicit surface function $\ic(\state)$: $\targetset = \{\state: \ic(\state) \le 0\}$. The value function $\brsvf(T, \state)$ is the implicit surface function representing the BRS: $\brs(T) = \{\state: \brsvf(T, \state) \le 0\}$.
%The Hamiltonian $\ham$ depends on the system dynamics, and the roles of control inputs. When introducing the different variants of BRSs, we will also introduce the corresponding Hamiltonian at the same time. 

In a similar fashion, the FRS $\frs(T)$ can be computed with the following HJI PDE:
\begin{equation}
\label{eq:frs_pde}
\begin{aligned}
D_t \frsvf(t,\state) + \ham(\state, \nabla \frsvf(t,\state)) = 0 \\ 
\frsvf(0,\state) = \ic(\state) \\
t\in [0,T]
\end{aligned}
\end{equation}

\noindent where the target set $\targetset = \{\state: \ic(\state) \le 0\}$ specifies the initial set of states, the FRS $\frs(T) = \{\state: \frsvf(T, \state) \le 0\}$ is given by the value function, and the Hamiltonian $\ham$ depends on the role of the control inputs.

\subsubsection{Potential Conflict\label{sec:SL}}
Intuitively, two vehicles are in potential conflict when they are likely to enter each other's danger zones. This indicates when a vehicle will employ the safety controller as opposed to using to the goal satisfaction controller. This notion is formally stated in the following definition of BRS $\brs^{PC}_{ij}(T)$ (PC for potential conflict), Hamiltonian, and optimal avoidance control for vehicle $\veh_i$:

\begin{defn}
    \label{defn:pot_conf}
    The \textbf{Potential Conflict Set}, $\PCS_{ij}$, is the set of relative states $\state_{ij}$ such that if $\state_i \notin \PCS_{ij}$, then $\veh_i$ is guaranteed to be able to avoid $\dz_{ij}$  with $\veh_j$, no matter what non-anticipative control strategy $\veh_j$ uses \cite{Mitchell05}.
    
    \begin{equation}
    \label{eq:pairwiseCA_brs}
    \begin{aligned}
    &\brs^{PC}_{ij}(T) = \{y: \forall \ctrl_i, \exists \ctrl_j, \exists t \in [0, T], \\
    &\quad\state_{ij}(\cdot) \text{ satisfies \eqref{eq:rdyn}}, \state_{ij}(t) = y \in \dz_{ij} \}
    \end{aligned}
    \end{equation}
    \begin{equation}
    \ham^{PC}_{ij}(\state_{ij}, \costate) = \max_{u_i} \min_{u_j} \costate \cdot \reldyn_{ij}(\state_{ij}, \ctrl_i, \ctrl_j)
    \end{equation}
    \begin{equation}
    \label{eq:ctrl_CA}
    \ctrl_i^{PC}(\state_{ij}) = \arg \max_{u_i} \min_{u_j} \nabla\brsvf_{ij}^{PC}(\state_{ij}) \cdot \reldyn_{ij}(\state_{ij}, \ctrl_i, \ctrl_j)
    \end{equation}
    
    \noindent where $\brsvf_{ij}^{PC}(\state_{ij}) := \lim_{T\rightarrow\infty}\brsvf_{ij}^{PC}(\state_{ij}, T)$.
     
    \begin{equation}
    \PCS_{ij} = \{\state_{i}: \brsvf_{ij}^{PC}(\state_{ij}) \le \st\}
    \end{equation}
\end{defn}
%$\brs^{PC}_{ij}(t)$ can be obtained by solving the HJ PDE \eqref{eq:brs_pde} with the Hamiltonian
%After computing the value function $\brsvf_{ij}^{PC}(\state_{ij}, T)$, the BRS $\brs^{PC}_{ij}(T)$ is given by $\brs^{PC}_{ij}(T) = \{\state_{ij}: \brsvf_{ij}^{PC}(\state_{ij}, T) \le 0\}$. For simplicity, we will take $T \rightarrow \infty$ and assume that $\brsvf_{ij}^{PC}(\state_{ij}, T)$ converges, as in \cite{Mitchell05} and \cite{Chen2016}. We write 
%
%\begin{equation}
%\brsvf_{ij}^{PC}(\state_{ij}) = \lim_{T \rightarrow \infty} \brsvf_{ij}^{PC}(\state_{ij}, T)
%\end{equation}
%The optimal avoidance control for vehicle $\veh_i$ is then

When $N$ vehicles are in potential conflict, then an $N$-vehicle collision avoidance algorithm is needed to ensure their safety. 

%Algorithms exist for $N=2$ and $N=3$ \cite{Mitchell05, Chen16}; this paper investigates the situation involving \textit{four} vehicles by proposing an \mbox{$N+1$-vehicle} collision avoidance scheme given an $N$-vehicle collision avoidance algorithm.
% 
\begin{rmk}
    Unlike $\brs^{PC}_{ij}$, $\PCS_{ij}$ is defined in the space of $\state_i$ instead of $\state_{ij}$. This is done for convenience in latter parts of the paper.
\end{rmk}

%The minimum time to exit altitude range set is the set of states such that either vehicle, $\veh_{i}$ or $\veh_{j}$ can exit altitude range safety if both vehicles are unable to avoid. However, if a vehicle can actively avoid, then the vehicles will not enter $D_{ij}$. We want to compute the set of states such that if $\veh_{i}$ can avoid, then the relative state $x_{i}$ will not be $D_{ij}$ regardless of $\veh_{j}$'s actions. Therefore, we look at the following Hamiltonian: 

%\subsubsection{Pairwise Collision Avoidance \label{sec:pairwise_CA}}  
%When there are only two vehicles $\veh_i, \veh_j$, the following control scheme for $\veh_i$ guarantees that $\veh_i$ and $\veh_j$ do not enter each other's danger zones $\dz_{ij}$:
%
%\begin{equation}
%\ctrl_i \in
%\begin{cases}
%\{\ctrl_i^{PC}(\state_{ij}) \text{ in \eqref{eq:ctrl_CA}}\} \text{ if } \state_i \in \PCS_{ij} \\
%\cset \text{ otherwise}
%\end{cases}
%\end{equation}

%This result follows from the definition of $\PCS_{ij}$ and $\brs^{PC}_{ij}$.

%HJ reachability definition \\
%FRS definition \\
%Value function definition, safety values description \\
%Collision Set definition \\
%Potential Conflict definition \\

% !TEX root = nplus1.tex
\subsection{Goals of this Paper \label{sec:goals}}
%Given $N+1$ vehicles with dynamics \eqref{eq:vdyn}, relative dynamics \eqref{eq:rdyn}, and some $N$-vehicle collision avoidance algorithm, we would like to determine guaranteed safety properties of the $N+1$-vehicle system. Although our proposed method is applicable to any collision avoidance algorithm, for concreteness we will describe our method in the context of the methods described in Sections \ref{sec:pairwise_CA} and \ref{sec:MIP}.

Consider a scenario in which at time $t=0$, a set of $N+1$ vehicles are in potential conflict with each other; in other words, the ``total conflict size\footnote{The notion of conflict size via determining potential conflicts among multiple vehicles will be formally defined later in Definition \ref{defn:conf_size}.}'' is $N+1$. Since the $N$-vehicle collision avoidance algorithm can guarantee safety for $N$ of the $N+1$ vehicles, we can choose $N$ vehicles to be handled by such a collision avoidance algorithm, and derive conditions on the remaining vehicle under which the entire system of $N+1$ vehicles can remain safe. To account for the situation in which safety of the entire system is impossible, we will assume that within an ``exit time'' $T_e$, one of the vehicles will be able to remove itself from the system.

Given the vehicle dynamics in \eqref{eq:vdyn}, some joint objective, a previously chosen $N$-vehicle collision avoidance algorithm, the derived relative dynamics in \eqref{eq:rdyn}, and the danger zones $\dz_{ij}$, we propose a cooperative safety control strategy that performs the following:

\begin{enumerate}
\item Detect potential conflicts based on the joint configuration of $N+1$ vehicles, and in particular determine the size of conflict;
\item Conservatively determine without significant computation cost when all $N+1$ vehicles are guaranteed to be safe, and provide safety controllers for all $N+1$ vehicles in this case;
\item Determine when safety cannot be guaranteed, in which case a vehicle needs to be removed from the system.
\end{enumerate}

We will show that our analysis provides with low overhead a conservative approximation of various slices of a high-dimensional BRS representing the unsafe region in the joint state space of all $N+1$ vehicles. Such a high-dimensional BRS is intractable to compute directly. %Our algorithm also provides the unsafe regions for every possible "outsider vehicle" and if any one of these vehicles does not have an unsafe region, then we are guaranteed safety. We prove that our proposed control strategy guarantees safety if there are no unsafe regions. This approach can be broadly extended to $N+1$ vehicles schemes that have a previously chosen collision avoidance scheme to guarantee mutual safety for $N$ vehicles and desire safety guarantees for $\veh_{N+1}$. 
% Problem formulation (0.5p)
%For N vehicles in conflict that have mutual guaranteed safety and will not be *in collision* for an (not general - infinite time horizon) we want to guarantee safety for an N+1 vehicle from the other N vehicles for some time t by determining a safe initial set of locations.
%
%- Better motivation: refer to pairwise collision avoidance, but can't do for more than 2 \\  
%- Also have control synthesis - also how to stay out of set \\
%Need to add Assumptions \\

% !TEX root = nplus1.tex
\section{Methodology}
\label{sec:method}

In this section, we illustrate our proposed hybrid framework that ensures safety of $N+1$ vehicles. We will first define terminologies essential to our hybrid framework. 

% It turns out that a large amount of information can be gained through our method without significant computation burden; in fact, depending on the computation resources available, the computation can most likely be done on-board in real-time. In addition, the small amount of computation required to quickly check the safety of the joint system can also provide the safety-preserving controllers for all vehicles.

\begin{defn}
    \textbf{$T_e$-buffer set.}
    \begin{equation}
    \begin{aligned}
    &\brs_{ij}^{E} = \{y: \exists \ctrl_i, \exists \ctrl_j,   \exists t \in [0, T_e], \\
    &\quad\state_{ij}(\cdot) \text{ satisfies \eqref{eq:rdyn}}, \state_{ij}(t) = y \in \dz{ij}\}
    \end{aligned}
    \end{equation}
\end{defn}

The purpose of the $T_e$-buffer set is to account for the situation in which our method indicates that $N+1$-vehicle potential conflicts cannot be guaranteed to be resolved safely. The $T_e$-buffer set allows a vehicle a duration of $T_e$ to remove itself safely from the system; this can be done by maneuvers such as exiting the altitude range. To provide the last-resort option of removing a vehicle from the system, we specify $\dz_{ij}$ to be the $T_e$-buffer set. 

Therefore, if $\state_i$ is on the boundary of $\PCS_{ij}$, then $\veh_i$ is said to be ``in potential conflict'' with $\veh_j$, and the control $\ctrl_i^{PC}(\state_{ij})$ must be used by $\veh_i$ in order to ensure that the relative state of $\veh_i$ and $\veh_j$ does not enter the $T_e$-buffer set.

\begin{defn}
    \label{defn:conf_size}
    \textbf{Conflict size.}
    Consider a graph in which each vehicle $\veh_i$ is a node. We connect two nodes $\veh_i,\veh_j$ with an edge if and only if $\veh_i$ is in potential conflict with $\veh_j$, vice versa, or both. The number of nodes with at least one edge in the resulting graph is defined as the conflict size.
\end{defn}

% As an example, suppose we have four vehicles, $\veh_1$, $\veh_2$, $\veh_3$, $\veh_4$ with the following pairs of vehicles in potential conflict: $(\veh_1, \veh_3),\ (\veh_1, \veh_2),\ (\veh_2, \veh_3),\ (\veh_3, \veh_4),\ (\veh_4, \veh_3),\ (\veh_4, \veh_1)$. Then, there would be an edge between $\veh_1,\veh_2$, between $\veh_2,\veh_3$, between $\veh_3,\veh_4$, and between $\veh_4, \veh_1$ as shown in Figure \ref{fig:potential_conflict}. 

% \begin{figure}
% \centering
% \includegraphics[width=0.4\columnwidth]{fig/potential_conflict.pdf}
% \caption{4-way Potential Conflict Diagram.}
% \label{fig:potential_conflict}
% \end{figure}

% \begin{defn}
%     \textbf{Vehicle partition.}
%     Suppose that $N$ of the $N+1$ vehicles are avoiding each other according to an $N$-vehicle collision avoidance algorithm. We partition the set of vehicles into the set of $N$ vehicles, whose indices are denoted $\Nveh$, and the ``outsider'' vehicle, $\veh_\OI$, whose index is denoted $\OI$.
% %    For example, if there are four vehicles $\veh_1, \ldots, \veh_4$, and $\veh_1, \veh_2, \veh_4$ are engaged in $3$-way collision avoidance, then $\Nveh = \{1, 2, 4\}$, and $\OI = 3$. 
% \end{defn}
% 
\subsection{Safety of the Outsider}
In this paper, we assume that at $t=0$, the conflict size first becomes at least $N$. When this happens, $N$ of the vehicles will perform avoidance maneuvers according to the $N$-vehicle collision avoidance algorithm. Let $T_r$ be the time it takes for the $N$-vehicle collision avoidance algorithm to reduce the conflict size \textit{among these $N$ vehicles} by at least $1$. We call $T_r$ the ``time of resolution.'' Note that $T_r$ needs not be known \textit{a priori}, and may be determined at $t=0$.

\subsubsection{Unsafe Region}
We let $\Nveh$ denote the set of indices of the $N$ vehicles performing avoidance maneuvers according to the $N$-vehicle collision avoidance algorithm and let $\OI$ denote the index of the ``outsider vehicle.'' The $N$ vehicles whose indices are in the set $\Nveh$ are guaranteed to be safe with respect to each other. For the safety of all $N+1$ vehicles, it remains to ensure that $\veh_\OI$, the ``outsider vehicle'' not in $\Nveh$, is also able to avoid the $T_e$-buffer set of all other vehicles. The following proposition provides an intuitive condition under which the safety of all $N+1$ vehicles can be guaranteed.

\begin{prop}
    \label{prop:np1}
    If $\veh_i, i\in\Nveh$ are in potential conflict and engaged in an $N$-way collision avoidance, then all $N+1$ vehicles are guaranteed to be able to avoid the $T_e$-buffer set $\brs_{ij}^{E}(T_e)$ if $\veh_\OI$ is in potential conflict with a maximum one of the other $N$ vehicles during $t \in [0, T_r]$.
\end{prop}

\textit{Proof:} Since the $N$-way collision avoidance algorithm guarantees that no vehicle $\veh_i, i\in\Nveh$ enter each other's $T_e$-buffer set $\brs_{ij}^{E}(T_e)$, it remains to ensure that the outsider vehicle $\veh_\OI$ does not enter any of the other vehicle's $T_e$-buffer set. If $\veh_\OI$ is not in potential conflict with any vehicle, then it is safe. And if $\veh_\OI$ is in potential conflict with exactly one other vehicle, say $\veh_1$ without loss of generality, then applying the optimal control $\ctrl_\OI^{APC}(\state_{\OI1})$ would guarantee that $\brs_{\OI1}^{E}(T_e)$ is not entered. \hfill $\blacksquare$

Suppose that to resolve the $N$-way conflict, a $N$-way collision avoidance algorithm stipulates some trajectories $\state_i(t), i\in\Nveh, t \in [0, T_r]$. To conservatively approximate the safe regions of $\veh_\OI$, we must ensure that the outsider vehicle cannot be in potential conflict with more than one vehicles until the $N$-vehicle conflict is resolved. We now define the region in which $\veh_\OI$ is in potential conflict with multiple vehicles.

\begin{defn}
    \textbf{Outsider unsafe region.}
    \begin{equation}
    \label{eq:our}
    \begin{aligned}
    \our_\OI(t) = \bigcup_{i,j \in \Nveh, i\neq j} \left(\bar\PCS_{\OI i}(t) \bigcap  \bar\PCS_{\OI j}(t)\right)   
    \end{aligned}
    \end{equation}
\end{defn}

Note that $\our_\OI(t)$ is a set of states in the space of $\state_i$, $\R^{n_i}$, and that $t \in [0, T_r]$.

If we can guarantee that $\veh_\OI$ never enters $\our_\OI(t)$ for $t \in [0, T_r]$, then the following control strategy would keep $\veh_\OI$ from entering the $T_e$-buffer sets of all other vehicles:

\begin{equation}
\label{eq:one_conf_ctrl}
\ctrl_\OI(\state_\OI, t) \in
\begin{cases}
\{\ctrl_\OI^{APC}(\state_{\OI j}(t))\} \text{ if } \state_\OI \in \bar\PCS_{\OI j}(t) \\
\cset \text{ otherwise}
\end{cases}
\end{equation}

For convenience, we will denote the set of control functions satisfying the control given in \eqref{eq:one_conf_ctrl} for $t \in [0, T_r]$ to be $\cfset_\OI^\text{SC}$ (SC for single conflict). Thus if $\ctrl_\OI(\cdot) \in \cfset_\OI^\text{SC}$, \textit{and} $\state_\OI(t) \notin \our_\OI(t) ~ \forall t\in[0,T_r]$, then  $\veh_\OI$ would be guaranteed to not enter into the $T_e$-buffer zones of any other vehicle. We now proceed to describe how to ensure the above conditions are met.

\subsubsection{Optimally Avoiding the Unsafe Region}
For $\veh_\OI$ to avoid $\our_\OI(t)$, we must ensure that $\veh_\OI$ avoids the set of states that inevitably leads to $\our_\OI(t), t \in [0, T_r]$. This set is captured by the following BRS:
\begin{defn}
    \textbf{Minimal BRS from $\our_\OI(t)$.}
    \begin{equation}
    \label{eq:brs_min}
    \begin{aligned}
    &\brs_\OI^-(t) = \{y: \forall \ctrl_\OI(\cdot) \in \cfset_\OI^\text{SC}, \exists s \in [t, T_r], \\
    &\quad\state_\OI(\cdot) \text{ satisfies \eqref{eq:vdyn} }, \state_\OI(s) = y \in \our_\OI(s)\}
    \end{aligned}
    \end{equation}
\end{defn}
This BRS can be computed using the $\our_\OI(t)$ as a time-varying target by solving \eqref{eq:brs_pde} with the Hamiltonian
\begin{equation}
\ham_\OI^-(t, \state_\OI, \costate) =
\begin{cases}
\costate \cdot \vdyn(\state_\OI, \ctrl_\OI^{APC}(\state_{\OI j}(t))) \text{ if } \state_\OI \in \bar\PCS_{\OI j}(t)\\
\max_{\ctrl_\OI \in \cset_\OI} \costate \cdot \vdyn(\state_\OI, \ctrl_\OI) \text{ otherwise}
\end{cases}
\end{equation}
%The implicit surface function $\brsvf_\OI^-(\state_\OI, t)$ representing $\brs_\OI^-$ gives the optimal control for avoiding $\our_\OI(t)$ while also avoiding any $T_e$-buffer sets when needed:
\begin{equation}
\label{eq:ctrl_minus}
\ctrl_\OI^-(\state_\OI, t) = \arg \max_{\ctrl_\OI \in \cset_\OI} \nabla \brsvf_\OI^-(\state_\OI, t) \cdot \vdyn(\state_\OI, \ctrl_\OI)
\end{equation}
%
%Using the definition of $\brsvf_\OI^-(\state_\OI, t)$ and the property of the optimal control derived from the value function, we can state the following:
\begin{prop}
    \label{prop:minimal}
     The following controller guarantees that if $\state_\OI(0) \notin \brs_\OI^-(0)$, then $\state_\OI(t) \notin \brs_\OI^-(t) ~ \forall t \in [0, T_r]$.
    
    \begin{equation}
    \label{eq:two_conf_ctrl}
    \ctrl_\OI(\state_\OI, t) \in
    \begin{cases}
    \{\ctrl_\OI^-(\state_\OI, t) \text{ in \eqref{eq:ctrl_minus}}\} \text{ if } \state_\OI \in \partial\brsvf_\OI^-(\state_\OI, t) \\
    \cset_\OI \text{ otherwise}
    \end{cases}
    \end{equation}
\end{prop}

Proposition \ref{prop:minimal} not only provides a safety-preserving controller for $\veh_\OI$. Along with the definition of $\brsvf_\OI^-(\state_\OI, t)$, but also represents a conservative guaranteed-safe region for $\veh_\OI$, given the avoidance maneuvers of $\veh_i, i \in \Nveh$. Depending on the system dynamics and computation power available, if $\brsvf_\OI^-(\state_\OI, t)$ is computable in real-time, then $\brsvf_\OI^-(\state_\OI, t)$ can be used to indicate a ``fenced off'' region that $\veh_\OI$ must avoid in order to ensure the safety of the $(N+1)$-vehicle system.

\subsubsection{Fast Safety Verification}
\label{sec:fast_ver}
When online computation resources are scarce, then we can consider the following BRS instead.

\begin{defn}
    \textbf{Maximal BRS from $\our_\OI$.}
    \begin{equation}
    \label{eq:brs_max}
    \begin{aligned}
    &\brs_\OI^+(t) = \{y: \exists \ctrl_\OI(\cdot), \exists s \in [t, T_r], \\
    &\quad\state_\OI(\cdot) \text{ satisfies \eqref{eq:vdyn} }, \state_\OI(s) = y \in \our_\OI(s)\}
    \end{aligned}
    \end{equation}
\end{defn}

Note that since the only difference between $\brs_\OI^-(t)$ and $\brs_\OI^+(t)$ is the quantifier in front of $\ctrl_\OI$, we have $\brs_\OI^-(t) \subseteq \brs_\OI^+(t)$, and therefore staying out of $\brs_\OI^+(t)$ is sufficient for guaranteeing safety.

Although computing $\brs_\OI^+(t)$ in general is just as difficult as computing $\brs_\OI^-(t)$, checking whether a state $\state_\OI$ is in $\brs_\OI^+(t)$ can be done much faster, without actually computing $\brs_\OI^+(t)$. This is achieved through the following FRS in $t \in [0, T_r]$:

\begin{defn}
    \label{defn:frs}
    \textbf{Forward reachable set from $\state_\OI$.}
    \begin{equation}
    \label{eq:frs}
    \begin{aligned}
    &\frs_\OI(t)  = \{y: \exists \ctrl_\OI(\cdot), \state_\OI(\cdot) \text{ satisfies \eqref{eq:vdyn}}, \\ &\quad\state_\OI(0) = \state_{\OI 0}, \state_i(t) = y \}
    \end{aligned}
    \end{equation}
\end{defn}

% $\frs_\OI(t)$ can be computed using $\eqref{eq:frs_pde}$ with the Hamiltonian
% 
% \begin{equation}
% \ham_\OI^\frs(\state_\OI, \costate) = \max_{\ctrl_\OI} \costate \cdot \vdyn_\OI(\state_\OI, \ctrl_\OI)
% \end{equation}
% 
This FRS can be computed offline assuming that $\state_{\OI 0}$ is at the origin. Vehicle dynamics usually do not depend on the position of the vehicle. In this case, the FRS can be simply transformed to the coordinate system in which the vehicle is at the origin of the state space. Hence, no significant online computation is required. We now state the equivalence of checking the membership of $\state_{\OI 0}$ in the BRS $\brs_\OI^+(0)$ and checking membership of $\state_{\OI 0}$ in the FRS $\frs_\OI(t), t \in [0, T_r]$.

\begin{prop}
    \label{prop:frs_brs}
    \begin{equation}
    \state_\OI(0) \in \brs_\OI^+(0) \Leftrightarrow \exists t \in [0, T_r], \frs_\OI(t) \cap \our_\OI(t) \neq \emptyset
    \end{equation}
\end{prop}

\begin{proof} 

Suppose $\state_\OI(0) \in \brs_\OI^+(0)$, then equivalently
\begin{equation}
\exists \ctrl_\OI(\cdot), \exists t \in [0, T_r], \state_\OI(\cdot) \text{ satisfies \eqref{eq:vdyn}}, \state_\OI(t) = y \in \our_\OI(t)
\end{equation}
Since all quantifiers are existential, we equivalently have
\begin{equation}
\exists t \in [0, T_r], \exists \ctrl_\OI(\cdot), \state_\OI(\cdot) \text{ satisfies \eqref{eq:vdyn}}, \state_\OI(t) = y \in \our_\OI(t)
\end{equation}
By Definition \ref{defn:frs}, we in turn equivalently have
\begin{equation}
\exists t \in [0, T_r], y \in \frs_\OI(t)
\end{equation}
But we also have $y \in \our_\OI(t)$. Therefore, equivalently, $\frs_\OI(t) \cap \our_\OI(t) \neq \emptyset$.
\end{proof}

\subsection{Hybrid Framework}
\label{sec:hybrid}
Based on the reachability analysis of outsider safety, we propose a hybrid framework that determines control for the $N+1$ vehicles depending on the conflict size. Our proposed framework provides several stages of safety checking. The framework is depicted in Fig. \ref{fig:main_hybrid}.

Note that all the reachable sets are either defined in the relative coordinates between two vehicles, or in the absolute coordinate with respect to one vehicle, and thus the system dimensionality remains low, allowing for tractable HJ reachability computations. In addition, with the exception of $\brs_\OI^-(t)$, the minimal BRS from $\our_\OI(t)$, all computations can be done offline and the resulting reachable sets can be stored as lookup tables. This way, we do not incur significant online computational burden.

\begin{figure}
\centering
\includegraphics[width=\columnwidth,trim={0.1cm 0.1cm 0.1cm 0.1cm},clip]{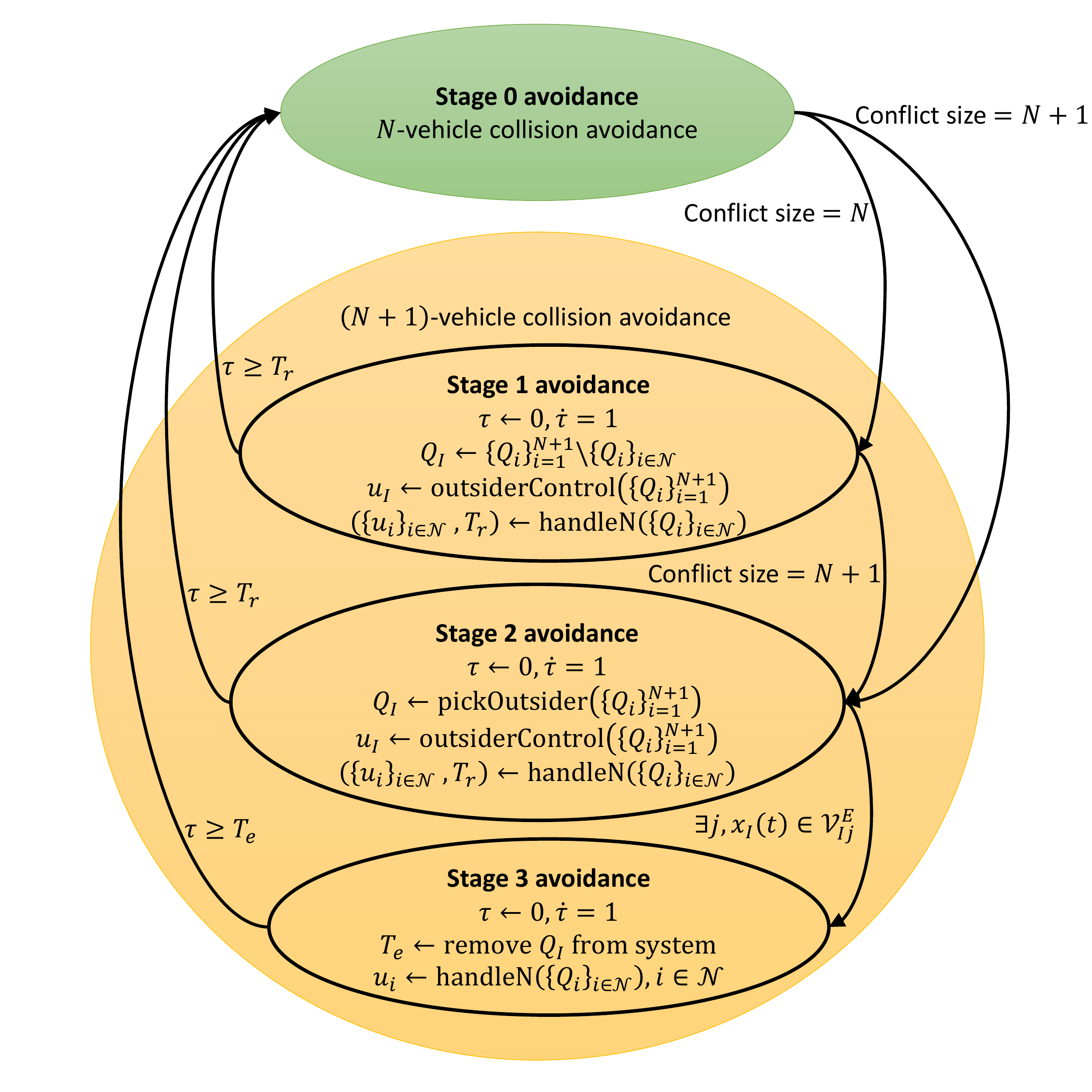}
\caption{Summary of Hybrid Framework. The stages of avoidance are explained in Section \ref{sec:hybrid}.}
\label{fig:main_hybrid}
\end{figure}
\subsubsection{Stage 0 Avoidance}
When the conflict size is less than $N$, the vehicles in potential conflict can be managed by the $N$-vehicle collision avoidance algorithm, and all vehicles are guaranteed safe. 

\subsubsection{Stage 1 Avoidance}
When the conflict size is $N$, we have a clear outsider vehicle $\veh_\OI$ that is not in conflict with other vehicles $\veh_i, i\in\Nveh$. Our hybrid controller preemptively attempts to keep $\veh_\OI$ from being in potential conflict with more than one of the other $N$ vehicles during $t \in [0, T_r]$. Algorithm \ref{alg:outsider_vehicle} provides a safety controller strategy for $\veh_\OI$. 

\subsubsection{Stage 2 Avoidance}

During Stage 2 avoidance, the size of conflict is $N+1$. To attempt to ensure safety, we determine whether there exists a vehicle that can be designated as $\veh_\OI$, such that it is guaranteed to only need to avoid one of the other $N$ vehicles for a duration of $T_r$. Algorithm \ref{alg:pickOutsider} outlines the process of designating a vehicle as the outsider vehicle in detail. After the outsider is designated, the safety control for $\veh_\OI$ is given according to Algorithm \ref{alg:outsider_vehicle}. 

If $\veh_\OI$ does enter $\our_\OI(t)$ during $t \in [0, T_r]$, no guarantees on avoiding the $T_e$-buffer sets can be made; however, $\veh_\OI$ may be able to avoid multiple vehicles with the same control. If at time $t$ before $T_r$, $\state_\OI(t) \in \brs_{\OI j}^{E}$, then we enter Stage 3 avoidance. Else, at $T_r$, we enter Stage 0.

\subsubsection{Stage 3 Avoidance}

$\veh_\OI$ is removed from the system, and the total number of vehicles is reduced to $N$, allowing the $N$-vehicle collision avoidance algorithm to maintain safety. The removal of the vehicle can be done by, for example, exiting the altitude range.

\subsection{Outsider Vehicle Control Strategy}
\label{sec:outCtrl}
\begin{algorithm}
\begin{algorithmic} 
	\caption{$\mathsf{outsiderControl}(\veh_{1}, \ldots, \veh_{N+1})$}
    \label{alg:outCtrl}
	\label{alg:outsider_vehicle}
	\STATE totalVehicleSet = $\{\veh_{1}, \ldots, \veh_{N+1}\}$
	\STATE nVehicleSet =  totalVehicleSet - $\veh_\OI$
	\STATE $[T_r, NTraj]$ = handleN(nVehicleSet)
	\STATE $\our_\OI$ = computeUnsafeRegions($T_r$, NTraj)
	\IF{$\our_\OI \cap \frs_\OI(t) == \emptyset\ \forall t \in [0, T_r]$}
		\STATE Return $\ctrl_\OI(\state_\OI, t)$ from \eqref{eq:one_conf_ctrl}
	\ELSE
		\STATE $\brs_\OI^-(T_r)$ = getBRS($\our_\OI$)  
		\STATE Return $\ctrl_\OI(\state_\OI, t)$ from \eqref{eq:two_conf_ctrl}	
	\ENDIF	
\end{algorithmic}
\end{algorithm}

The $N+1$ collision avoidance algorithm is guaranteed to be safe if $\veh_\OI$ is guaranteed to be able to avoid $\our_\OI(t)$. To determine $\our_\OI(t)$, the trajectories $\state_i(t), i \in \Nveh, t\in [0, T_r]$ are needed. We assume that these trajectories along with $T_r$ are given by the $N$-vehicle collision avoidance algorithm, which we call ``handleN()'' in Algorithm \ref{alg:outCtrl}.

Algorithm \ref{alg:outCtrl} first tries the fast safety verification method outlined in Section \ref{sec:fast_ver} to check if $\veh_\OI$ is guaranteed to avoid $\our_\OI(t)$. If $\veh_\OI$ is outside the maximal BRS, $\brs_\OI^+(T_r)~\forall t\in [0, T_r]$, then $\veh_\OI$ is guaranteed safe. As proved earlier, this can be efficiently checked using the FRS: If $\frs_\OI(t)$ does not intersect with $\our_\OI(t) ~ \forall t \in [0, T_r]$, then $\veh_\OI$ is guaranteed to be safe. By the definition of $\frs_\OI(t)$, regardless of $\veh_\OI$'s actions, it will not need to avoid more than one vehicle in the set $\{\veh_i\}, i \in \Nveh$. In this case, \eqref{eq:one_conf_ctrl} specifies the control strategy for $\veh_\OI$. 

If there is an intersection between $\frs_\OI(t)$ and $\our_\OI(t)$ for some $t \in [0, T_r]$, we can still check if $\veh_\OI$ is outside the minimal BRS $\brs_\OI^-(0)$ from $\our_\OI(t)$. $\brs_\OI^-(0)$ specifies the set of states outside of which $\veh_\OI$ is guaranteed to have a control strategy to not enter $\our_\OI(t)$. If $\state_\OI(0) \notin \brs_\OI^-(0)$, then \eqref{eq:two_conf_ctrl} specifies the safety-preserving control strategy for $\veh_\OI$. 

If $\state_\OI(t) \in \brs_\OI^-(0)$, then $\veh_\OI$ is not guaranteed to avoid $\our_\OI(t)$. In this case, it may be possible for $\veh_\OI$ to enter a $T_e$-buffer zone, in which case $\veh_\OI$ will be safely removed from the system.

\subsection{Designating a Suitable Outsider}

If the system enters Stage 2 avoidance, a vehicle needs to be designated as the outsider vehicle $\veh_\OI$. This is achieved by iterating through each vehicle $\veh_i$ where $i\in \{1 \dotsc {N+1}\}$ and determining whether safety guarantees can be made if $\veh_i$ is the outsider vehicle. Algorithm \ref{alg:pickOutsider} describes this process in detail.

\begin{algorithm}
\begin{algorithmic}
	\caption{$\mathsf{pickOutsider}(\veh_{1}, \ldots, \veh_{N+1})$   \label{alg:pick_outsider}}
	\label{alg:pickOutsider}
	\STATE totalVehicleSet = $\{\veh_{1}, \ldots, \veh_{N+1}\}$
	\STATE leastConflictVehicles = getLeastConflictVehicles()
	\FOR{$\veh_\OI \in \text{leastConflictVehicles}$}
		\STATE nVehicleSet =  totalVehicleSet - $\veh_\OI$
		\STATE $[T_r, NTraj]$ = handleN(nVehicleSet)
		\STATE $\our_\OI$ = computeUnsafeRegions($T_r$, NTraj)
		\IF{$\our_\OI \cap \frs_\OI(t) == \emptyset\ \forall t \in [0, T_r]$}
			\STATE Return $\veh_\OI$
		\ELSE
			\STATE $\brs_\OI^-(T_r)$ = getBRS($\our_\OI$) 
			\IF{$\veh_\OI\ \not \in\ \brs_\OI^-(T_r)$} 
				\STATE Return $\veh_\OI$	
			\ENDIF					
		\ENDIF	
	\ENDFOR
	\STATE Return any $\veh_i, i\in \{1 \dotsc {N+1}\}$	
\end{algorithmic}
\end{algorithm}

While picking a suitable outsider vehicle, we want to have the $N$ vehicle collision avoidance algorithm resolve as many conflicts as possible. Therefore, we select the vehicle with the least conflicts with the other vehicles. After the exclusion of this vehicle, the $N$-vehicle collision avoidance algorithm can resolve the conflicts of the rest of the $N$ vehicles.

Since we have chosen the outsider vehicle to be the one with the least number of conflicts, the remaining $N$ vehicles will still have a conflict size of at least one, and the $N$-vehicle collision avoidance algorithm would then reduce this conflict size by at least $1$ within a duration of $T_r$.

\subsection{Communication Scheme}
The $N+1$ collision avoidance algorithm does not require a centralized communication architecture. For each possible outsider $\veh_\OI$, as long as the trajectory of each of the $\veh_i$ where $i \in \Nveh$ is broadcast, $\veh_\OI$ can determine if safety is guaranteed. Once an outsider $\veh_\OI$ is determined, its control strategy can be determined with the broadcast trajectories. Note that a protocol needs to be established to avoid assigning multiple outsider vehicles; for example, the first vehicle to declare itself as an outsider could be chosen, as long as it is a valid choice according to Algorithm \ref{alg:pick_outsider}.

% Solution methodology (1.5-2p)
%% HJ Reachability
%% Mixed integer program

% !TEX root = nplus1.tex
\section{Safety Guarantee}
\label{sec:proofs}

In the previous section, we proposed a hybrid framework that allows us to detect whether the system of $N+1$ vehicles is guaranteed to be safe. We now formally state this guarantee and prove the result.

\begin{thm}
\label{thm:main_result}
Suppose there are $N$ vehicles that enter into potential conflict at time $t=0$. Denote the outsider vehicle as $\veh_\OI$. If $\veh_\OI$ does not enter $\our_\OI(t)$ for $t \in [0, T_r]$, then we are guaranteed safety for $N+1$ vehicles for a duration of $T_r$.
\end{thm}

\begin{proof}
To prove our theorem, it suffices to prove the following points: 
  
The set $\PCSOIJ$ refers to the locations of potential conflict between  $\veh_\OI$ and $\veh_j$. If we assume that $\veh_j$ is part of the N vehicle collision avoidance algorithm, then our safety controller is assumed to only guarantee safety when  $\veh_\OI$ can avoid one vehicle at time. Therefore, if $\veh_\OI$ is in the potential conflict set $\PCSOIJ$,  $\veh_\OI$ must avoid $\veh_{j}$ since $\veh_{j}$ is occupied with the $N$-vehicle collision avoidance algorithm. 

If $\veh_\OI$ is in the $P_{\OI j}$ and $P_{\OI a}$, it is in potential conflict with $\veh_{j}$ and $\veh_{a}$. If  $\veh_\OI$ is able to avoid both for $T_r$, then perhaps  $\veh_\OI$ can be guaranteed safety from both $\veh_{j}$ and $\veh_{a}$. However, it is not always guaranteed that $\veh_\OI$ can avoid both $\veh_{j}$ and $\veh_{a}$ with the same control. The set $\our_\OI(t)$ for $t \in [0, T_r]$ is the set of locations where $\veh_\OI$ will be in potential conflict with more than $1$ vehicle. Therefore, we can guarantee safety if $\veh_\OI$ does not enter $\our_\OI(t)$ for $t \in [0, T_r]$ because $\veh_\OI$ is guaranteed to be able to avoid at most one vehicle at a time. Note that we have effectively computed a conservative approximation of a very high-dimensional reachable set by determining unsafe regions.
\end{proof}

\begin{thm}
If $\veh_\OI$ is in $\our_\OI(t)$ for some $t \in [0, T_r]$, $\veh_\OI$ may still be able to avoid the other vehicles as long as it does not enter the $\brs_{ij}^E$ of any $\veh_{j}$. If $\veh_\OI$ enters $\brs_{\OI j}^E$ of any $\veh_{j}$, $\veh_\OI$ is forced to exit altitude range and is guaranteed to not collide with another vehicle. 
\end{thm}

\begin{proof}
For a particular configuration, if there does not exist a $\veh_\OI$ that can avoid $\brs_\OI^-(t)$ for all $t \in [0, T_r]$, then we are not guaranteed to have a control strategy where $N+1$ vehicles are guaranteed safe without forcing a vehicle to exit. 

However, even if we are unable to find an ideal vehicle $\veh_\OI$ that can avoid $\brs_\OI^-(t)$ for all $t \in [0, T_r]$, we can still employ a vehicle to be our $\veh_\OI$. At some point t, this $\veh_\OI$ will be in potential conflict with two vehicles, let's call them $\veh_{j}$ and $\veh_{k}$. It is possible that the avoidance control could avoid $\veh_{j}$ and $\veh_{k}$ at the same time. As long as $\veh_\OI$ does not enter $\brs_{\OI j}^E$ or $\brs_{\OI k}^E$, $\veh_\OI$ is not required to be removed from the system. 

If $\veh_\OI$ enters $\brs_{\OI j}^E$ or $\brs_{\OI k}^E$, then it must be removed from the system. By definition of $\brs_{ij}^E$, as long as $\veh_\OI$ is able to be removed from the system within a duration of $T_e$, safety can be guaranteed.
\end{proof}

\section{Numerical Simulations}
\label{sec:sim}

In this section, we illustrate our method through simulations of a four-vehicle system, building our $N+1$ vehicle framework on top of the three-vehicle collision avoidance MIP scheme. We assume that the dynamics of each vehicle $\veh_i$ are identical and described by the Dubins car model:
\begin{equation}
\begin{aligned}
	\dot p_{x,i} = v \cos(\theta_i) \\
	\dot p_{y,i} = v \sin(\theta_i) \\
	\dot \theta_i = \omega_i, \vert \omega_i \vert \le \bar \omega
\end{aligned}
\end{equation}

The states $p_{x,i}, p_{y,i}, \theta_i$ represent the $x$, $y$, and $\theta$ coordinates respectively of vehicle $\veh_i$. The danger zone for HJ computation between $\veh_i$ and $\veh_j$ is defined as
\begin{equation}
    \mathcal{Z}_{ij} = \{x_ij : (p_{x,i} - p_{x,j})^2 + (p_{y,i} - p_{y,j})^2 \le R_c^2 \}
\end{equation}

$x_{ij}$ refers to the relative coordinates between $\veh_i$ and $\veh_j$: $x_{ij} = [p_{x,ij}, p_{y,ij}, p_{\theta,ij}]^\top = [p_{x,i} - p_{x,j}, p_{y,i} - p_{y,j}, \theta_i - \theta_j]^\top$.

In our experiment, each vehicle travels at a fixed speed of $v = 1$ and its maximum turn rate is $\bar \omega = 1$. The time for a vehicle to exit $T_e$ is set to 2;  the collision radius $R_c$ for each danger zone $\mathcal{Z}_{ij}$ is 3; and the safety threshold $K$ is 2.0.

To obtain the optimal pairwise safety controller, we compute the BRS $\brs^{PC}_{ij}(T)$ with relative dynamics
\begin{equation}
\begin{aligned}
\dot p_{x,ij} = -v + v \cos(\theta_{ij}) + \omega_i p_{y,ij} \\
\dot p_{y,ij} = v \sin(\theta_{ij}) - \omega_i p_{x,ij}\\
\dot \theta_{ij} = \omega_i - \omega_j, \text{ }\vert \omega_i \vert, \vert \omega_j \vert \le \bar \omega
\end{aligned}
\end{equation}

We present the simulation for Stage 1 and 2 in detail when number of vehicles $N=4$. Each vehicle aims to reach the circular target of matching color while avoiding the danger zones of the other vehicles. 

\begin{figure}[]
  \centering
  \begin{subfigure}[b]{0.49\columnwidth}
    \includegraphics[width=\textwidth]{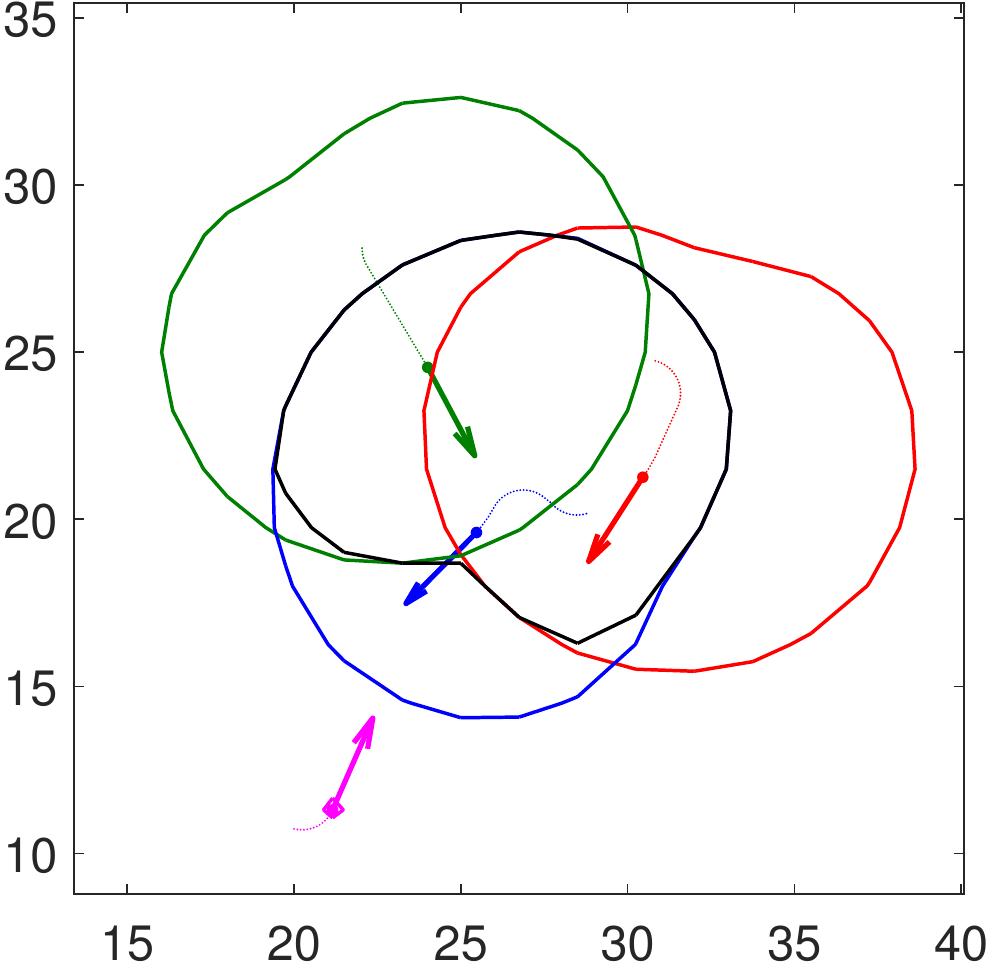}
  \end{subfigure}
  \begin{subfigure}[b]{0.49\columnwidth}
    \includegraphics[width=\textwidth]{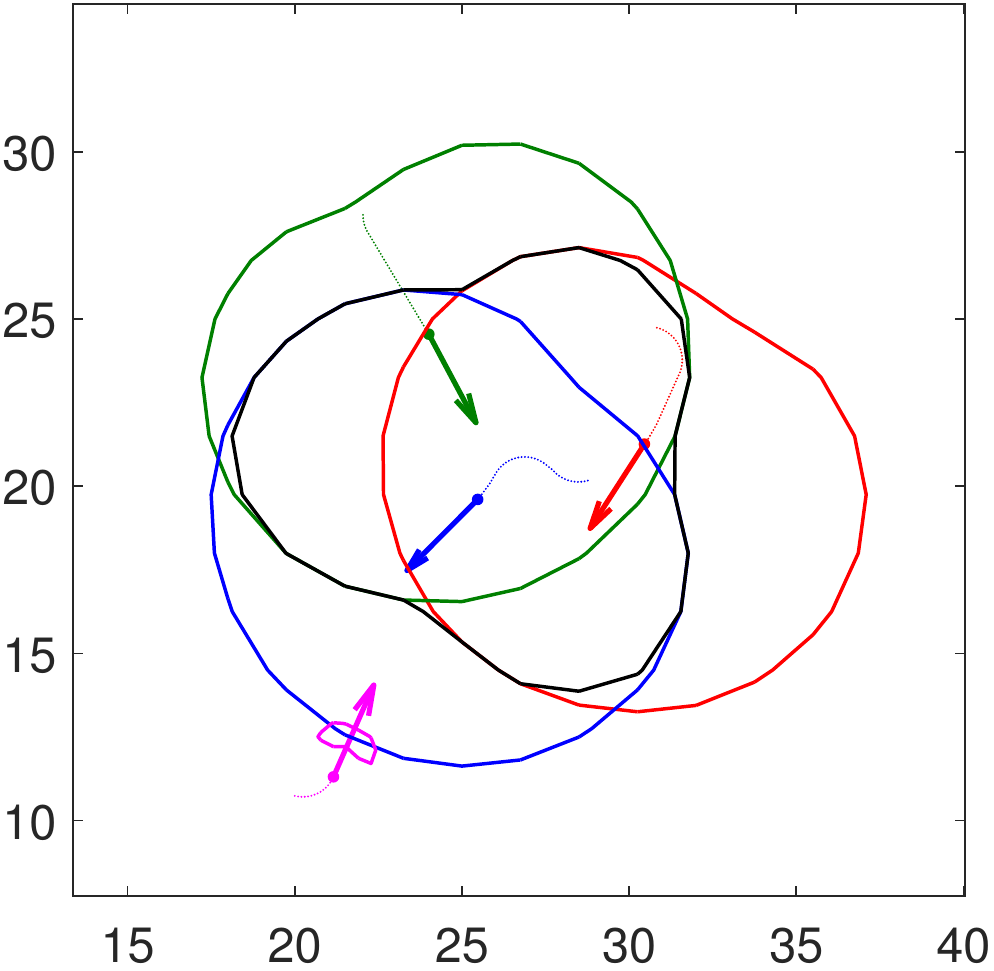}
  \end{subfigure}
  \caption{Stage 1: A scenario where the FRS of the outside vehicle (pink) does not intersect with the outsider unsafe region (black) for the entire time horizon that the rest of the three vehicles take to resolve their conflicts.}
  \label{fig:frs_nonoverlap}
\end{figure}

\begin{figure}[]
  \centering
  \includegraphics[width=0.5\columnwidth]{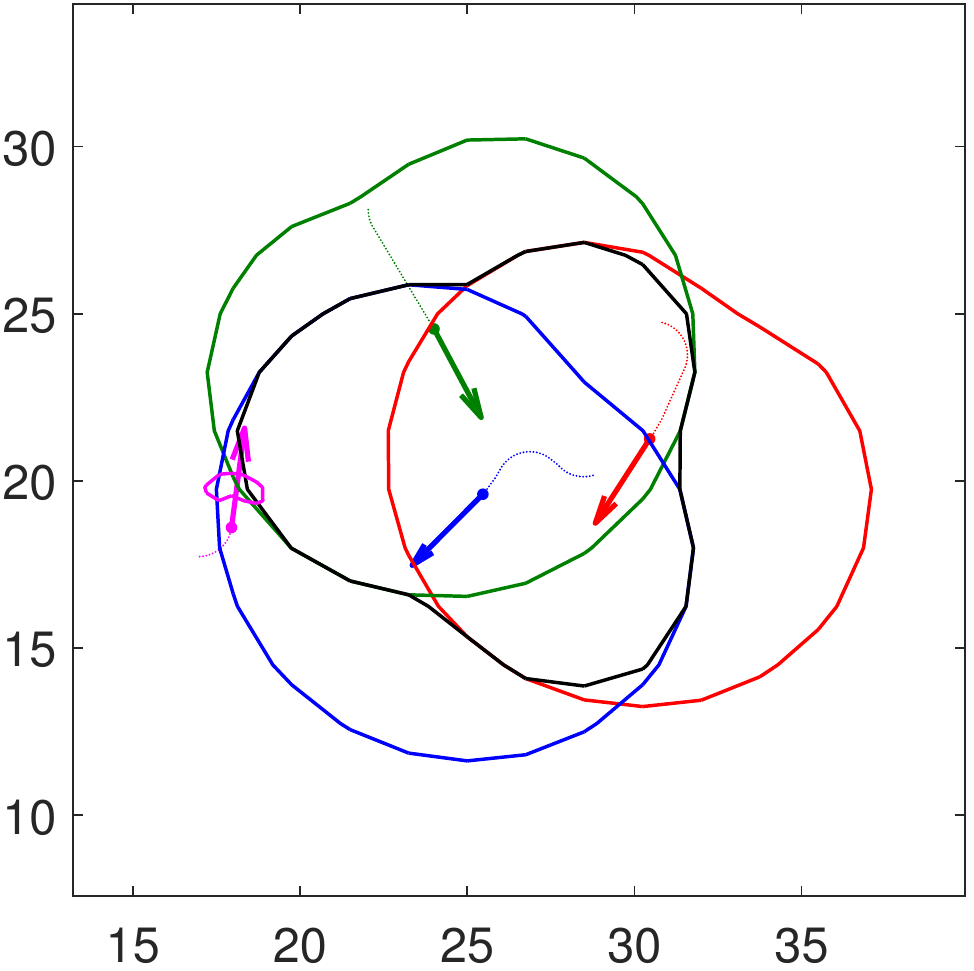}
  \caption{A stage 1 scenario where the FRS of the outside vehicle (pink) intersects with the outsider unsafe region (black).}
  \label{fig:frs_overlap}
\end{figure}

\begin{figure}[]
  \centering
  \includegraphics[width=0.5\columnwidth]{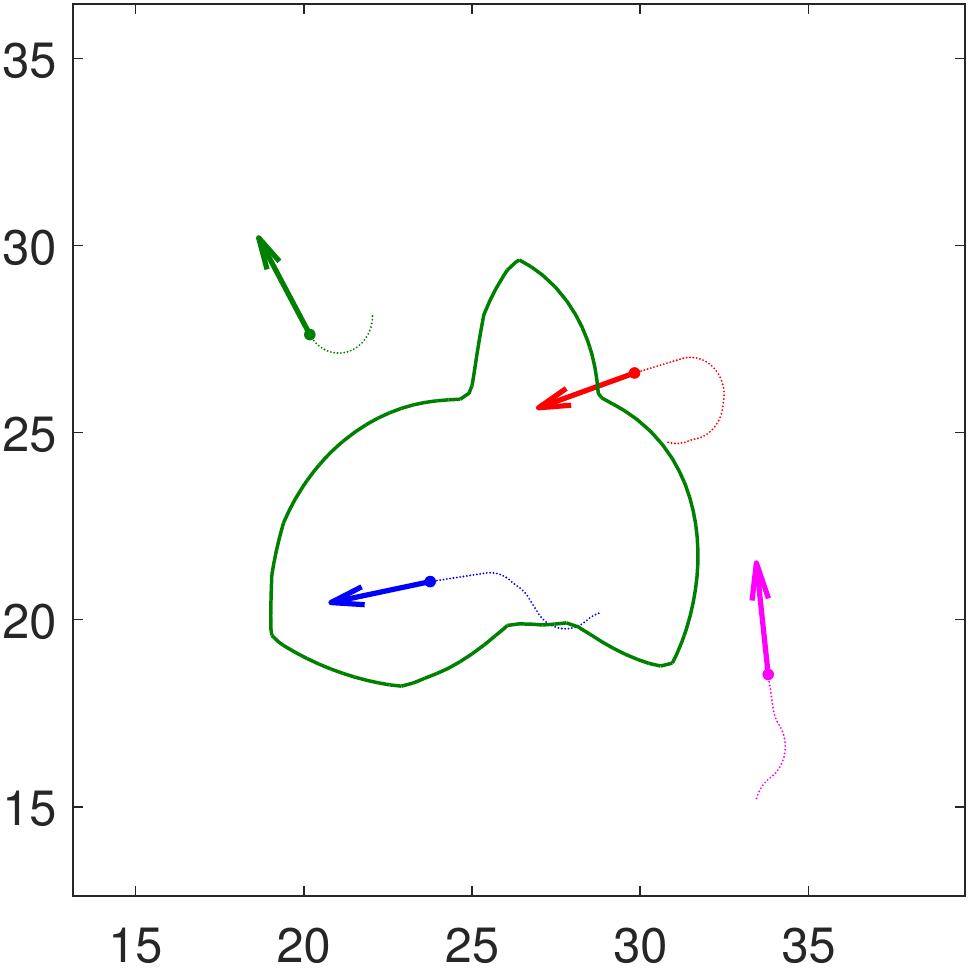}
  \caption{A stage 2 avoidance scenario in which the green vehicle is chosen as the outsider; the green boundary shows that it is outside of $\brs_\OI^-(t)$ and thus can avoid getting into multiple conflicts.}
  \label{fig:outsider_brs}
\end{figure}

In Fig. \ref{fig:frs_nonoverlap}, we illustrate the scenario where the FRS of the outsider vehicle (pink) does not overlap with the outsider unsafe region $\ref{eq:our}$ (black) in stage 1. The potential conflict zones are plotted for the three vehicles in their colors. In particular, we show two subplots illustrating the FRS and the outsider unsafe region at a time horizon of $T=2$ and $T=27$ respectively. In this case, the outsider vehicle is free to use the goal-satisfaction controller to go to its target while the rest of the three vehicles resolve their potential conflicts. In Fig. \ref{fig:frs_overlap}, we show a scenario where the FRS of the outsider vehicle intersects with the outsider unsafe region. In this case, we compute the BRS of this region and the outsider vehicle preemptively executes the optimal controller to avoid this region while the rest of the three vehicles resolve conflicts. Note that in this graph, we're taking the 2D projection of the reachable sets so it looks like the outsider vehicle is inside the potential conflict zone of another vehicle though the outsider vehicle is outside of the potential conflict zone in the state space.

Fig. \ref{fig:outsider_brs} shows a stage 2 scenario in which the green vehicle is chosen as the outsider. The green boundary represents the a 2D slice of the BRS $\brs_\OI^-(t)$ taken at the green vehicle's current heading, and contains the set of positions from which the green vehicle would inevitably enter the set $\our_\OI(t)$. Since the green vehicle is outside of $\brs_\OI^-(t)$, it is guaranteed to be able to avoid $\our_\OI(t)$, thereby avoid getting into multiple conflicts with the other three vehicles. Thus, if the green vehicle uses the controller $\ctrl_\OI^-(\state_\OI, t)$ in \eqref{eq:pairwiseCA_brs}, the entire conflict can be resolved, as shown in Fig. \ref{fig:sim}.

%In each plot of Figure \ref{fig:pick_outsider}, the vehicle marked with a $\textbf{*}$ is being chosen as the outsider vehicle. Each plot displays the $T_e$-buffer set, the augmented potential conflict set, the FRS of the outsider vehicle, and the BRS of the intersection of the $T_e$-buffer sets.
%
%In Figure \ref{fig:pick_outsider}a, there is an overlapping region between the FRS of the red outsider vehicle and the the intersection of the augmented potential conflict sets. Hence, the BRS set of the augmented potential conflict sets is generated for this case.

\begin{figure}
  \centering
  \begin{subfigure}[t]{0.45\columnwidth} \label{subfig:sim_1}
    \includegraphics[width=\columnwidth]{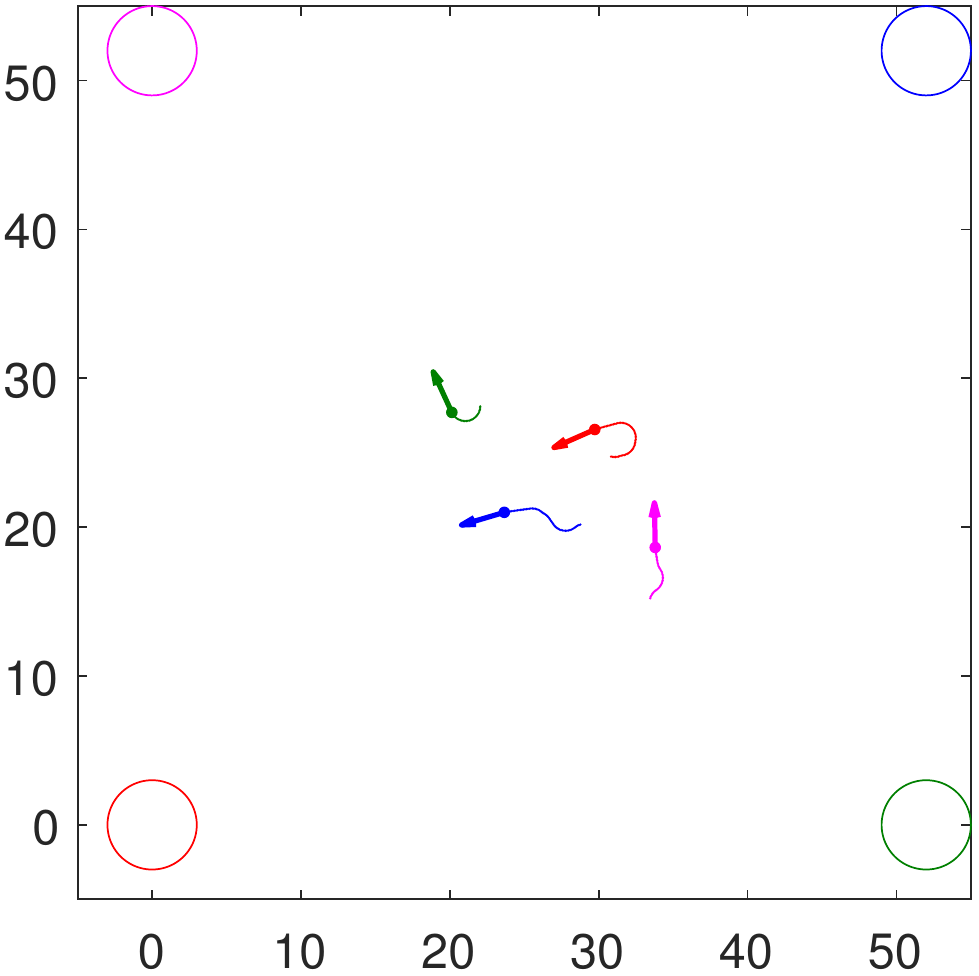}
    \caption{}
  \end{subfigure}
  \begin{subfigure}[t]{0.45\columnwidth} \label{subfig:sim_2}
    \includegraphics[width=\columnwidth]{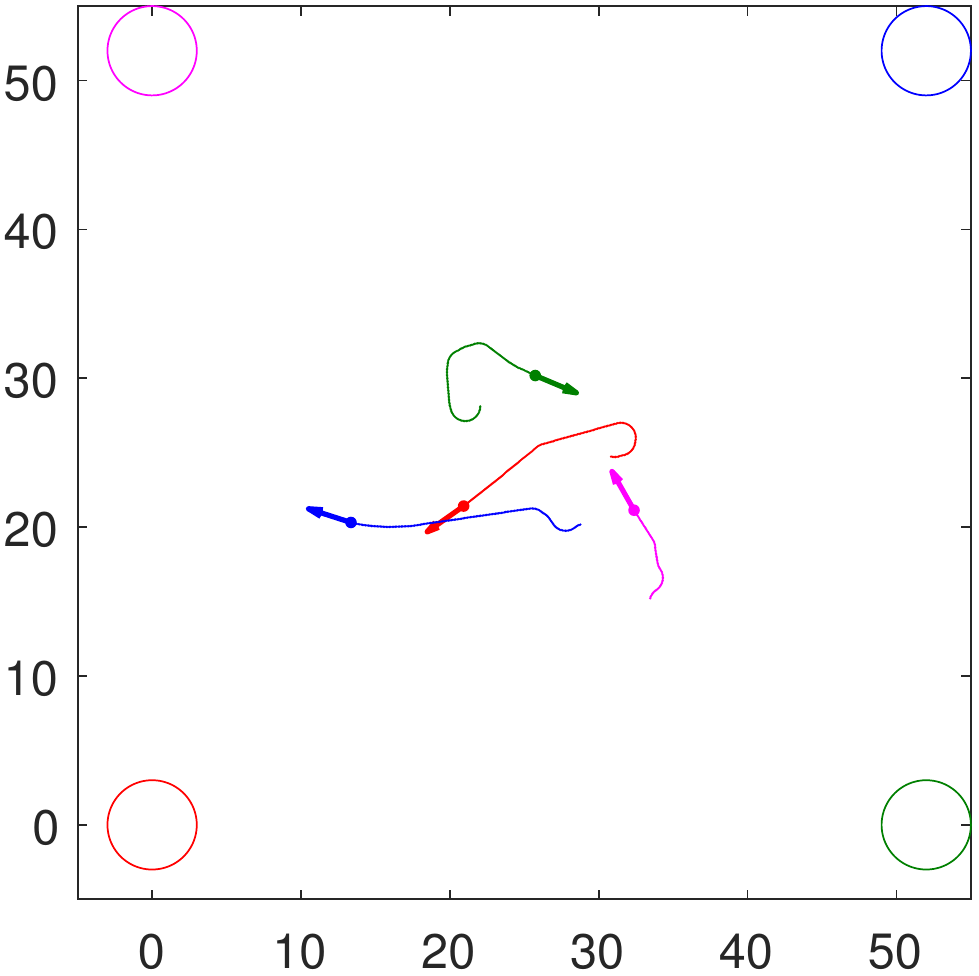}
    \caption{}
  \end{subfigure}
  
  \begin{subfigure}[t]{0.45\columnwidth} \label{subfig:sim_3}
    \includegraphics[width=\columnwidth]{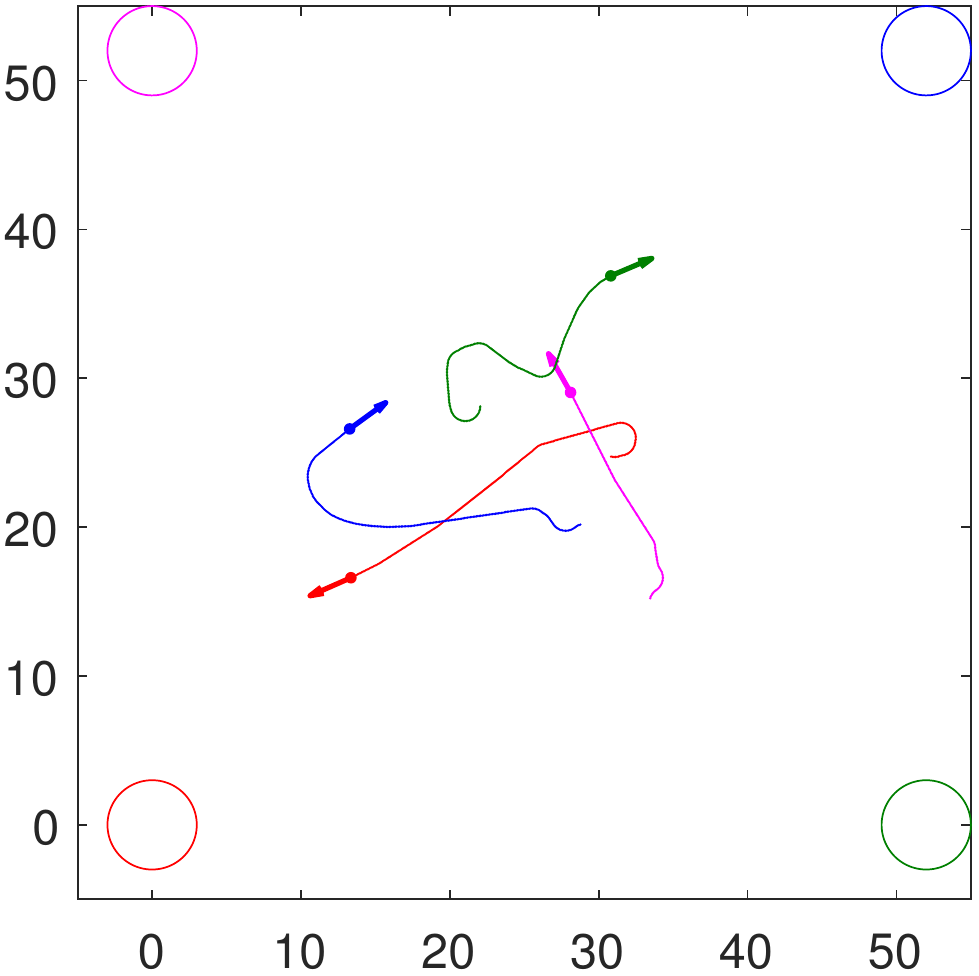}
    \caption{}
  \end{subfigure}
  \begin{subfigure}[t]{0.45\columnwidth} \label{subfig:sim_4}
    \includegraphics[width=\columnwidth]{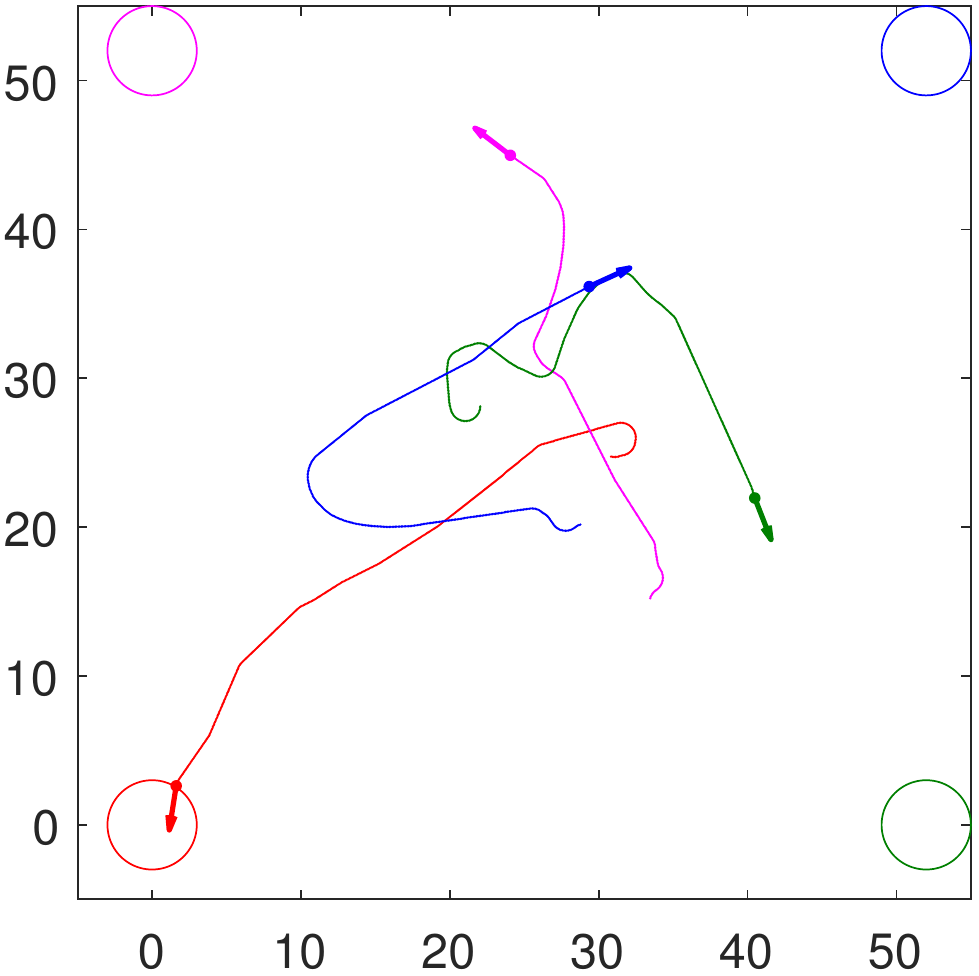}
    \caption{}
  \end{subfigure}   
  \caption{Snapshots in time of a four-vehicle simulation in which stage 2 avoidance occurs (the same simulation as that shown in Fig. \ref{fig:outsider_brs}). The green vehicle is chosen as the outsider and it does not interfere with the avoidance maneuvers of the other three vehicles. The colored circles represent the vehicles' destinations. \label{fig:sim}}
\end{figure}

The BRS of the $N$-vehicle potential conflict intersections is computed online. The other BRS and FRS can be computed offline and reused. Each BRS and FRS takes approximately 1 minute to compute. All computations were done on a MacBookPro 11.2 laptop with an Intel Core i7-4750 processor.

% Numerical Simulations (1-2p)

% !TEX root = nplus1.tex
\section{Conclusions}
We considered the problem of $N+1$-vehicle collision avoidance using a hybrid framework. By exploiting properties of pairwise optimal collision avoidance and an existing $N$-vehicle collision algorithm, our proposed method conservatively approximates the unsafe region of any vehicle given the states of the other $N$ vehicles, and synthesizes a cooperative control strategy that guarantees collision avoidance of all $N+1$ vehicles whenever possible as determined by our hybrid framework. In the situation in which collision avoidance cannot be guaranteed for all $N+1$ vehicles, our proposed method still allows enough time for any vehicle to remove itself from the system. 

Immediate future work includes better preemptive avoidance to allow prediction of the unsafe regions so that collision avoidance can be guaranteed in all scenarios. Other future work includes investigating safety guarantees for a larger number of vehicles, and incorporating more collision avoidance algorithms into our proposed framework.

% Conclusion (0.5p)

%%%%%%%%%%%%%%%%%%%%%%%%%%%%%%%%%%%%%%%%%%%%%%%%%%%%%%%%%%%%%%%%%%%%%%%%%%%%%%%%
%\addtolength{\textheight}{1cm}   % This command serves to balance the column lengths

 \bibliographystyle{IEEEtran}
 \bibliography{references}
\end{document}